\documentclass{amsart}
\usepackage{latexsym,amsmath,amssymb}
\newtheorem{thm}{Theorem}[section]
\newtheorem{cor}[thm]{Corollary}
\newtheorem{lem}[thm]{Lemma}

\newtheorem{pro}[thm]{Proposition}
\theoremstyle{definition}\newtheorem{defn}[thm]{Definition}
\theoremstyle{remark}
\newtheorem{rem}[thm]{Remark}
\numberwithin{equation}{section}

\begin{document}
	
\title
[WEIGHTED CONDITIONAL EXPECTATION]
{The ASCENT AND DESCENT OF WEIGHTED CONDITIONAL EXPECTATION OPERATORS }

\author{\sc\bf  Y. Estaremi }
\address{\sc  Y. Estaremi}

\email{estaremi@gmail.com, Y.Estaremi@qub.ac.uk}

\address{ Department of Mathematics, Payame Noor University (PNU), P. O. Box: 19395-3697, Tehran- Iran}




\begin{abstract}
In this paper we prove that the ascent and descent of a weighted conditional expectation operator of the form of $M_wEM_u$ on $L^p$-spaces under some weak conditions are finite and less or equal to 2. Moreover, we give some necessary and sufficient conditions for $M_wEM_u$ to be power bounded. In the sequel  we apply some results in operator theory on ascent and descent to $M_wEM_u$. Finally we find that $T=M_wEM_u$ is Cesaro bounded if and only if $\widehat{T}$ is Cesaro bounded.
\end{abstract}

\maketitle

\section{ \sc\bf Preliminaries and Introduction }

Let $X$ be a linear space and $T:X\longrightarrow X$  be a linear operator with domain $\mathcal{D}(T)$ and range $\mathcal{R}(T)$ in $X$.  The null space of the iterates of $T$, $T^n$, is denoted by   $\mathcal{N}(T^n)$, and we know that the null spaces of  $T^n$'s form an increasing chain of subspaces $
(0)=\mathcal{N}(T^0)\subset \mathcal{N}(T)\subset \mathcal{N}(T^2)\subset\ldots$. Also  the ranges  of iterates of $T$ form a nested chain of subspaces $X=\mathcal{R}(T^0)\supset \mathcal{R}(T)\supset \mathcal{R}(T^2)\supset\ldots$.
  Note that if  $\mathcal{N}(T^k)$ coincides with $\mathcal{N}(T^{k+1})$ for some $k$,  it coincides with all $\mathcal{N}(T^n)$ for $n>k$.
The smallest non-negative integer $k$ such that $\mathcal{N}(T^k)=\mathcal{N}(T^{k+1})$ is called the \textit{ascent} of $T$ and denotes by $\alpha(T)$. If there is no such $k$, then  we set $\alpha(T)=\infty$. Also  if $\mathcal{R}(T^k)=\mathcal{R}(T^{k+1})$, for some  non-negative integer $k$, then  $\mathcal{R}(T^n)=\mathcal{R}(T^k)$ for all $n>k$.
 The smallest non-negative integer $k$ such that $\mathcal{R}(T^k)=\mathcal{R}(T^{k+1})$ is called  \textit{descent} of $T$
and denotes by $\delta(T)$. We set $\delta(T)=\infty$ when there is no such $k$. When ascent and descent of an operator are finite, then they are equal and the linear space $X$ can be decomposed into the direct sum of the null and range spaces of a suitable  iterates of $T$. The ascent and descent of an operator can be used to characterized when an operator can be broken into a nilpotent piece and an invertible one; see, for example, \cite{abr, Tay}. For some results on ascent and descent of an operator in general setting see, for example, \cite{tay,yoo}.\\
The operator $T$ is called power bounded if the norms of the powers $T^k$, $k\geq0$, are uniformly bounded $(\sup_{k}\|T^k\|<\infty)$, and Cesaro bounded if the Cesaro means $A_n(T)=n^{-1}\sum^{n-1}_{i=0}T^{i}$ are uniformly bounded.

Let $(X, \mathcal{F}, \mu)$ be a complete $\sigma$-finite measure space. All sets and functions statements  are to be interpreted as holding up to sets of measure zero. We denote the
collection of (equivalence classes modulo sets of zero measure of)
$\mathcal{F}$-measurable complex-valued functions on $X$ by $L^0(\mathcal{F})$.
For a $\sigma$-subalgebra $\mathcal{A}$ of $\mathcal{F}$, the conditional expectation operator associated with $\mathcal{A}$ is the mapping $f\rightarrow E^{\mathcal{A}}f,$ defined for all non-negative function $f$ as well as for all $f\in L^p(\mathcal{F})=L^p(X, \mathcal{F}, \mu)$, $1\leq p\leq\infty$, where $E^{\mathcal{A}}f$ is the unique $\mathcal{A}$-measurable function satisfying

$$\int_{A}(E^{\mathcal{A}}f)d\mu=\int_{A}fd\mu \ \ \ \ \ \ \ \forall A\in \mathcal{A}.$$
We will often write $E$ for $E^{\mathcal{A}}$. This operator will play a major role in our
work and we list here some of its useful properties:

\vspace*{0.2cm} \noindent $\bullet$ \  If $g$ is
$\mathcal{A}$-measurable, then $E(fg)=E(f)g$.

\noindent $\bullet$ \ If $f\geq 0$, then $E(f)\geq 0$; if $E(|f|)=0$,
then $f=0$.

\noindent $\bullet$ \ $|E(fg)|\leq
(E(|f|^p))^{\frac{1}{p}}(E(|g|^{p'}))^{\frac{1}{p'}}$; $p^{-1}+p'^{-1}=1$.

\noindent $\bullet$ \ For each $f\geq 0$, $S(E(f))$ is the smallest $\mathcal{A}$-set containing $S(f)$, where $S(f)=\{x\in X: f(x)\neq 0\}$.

\vspace*{0.2cm}\noindent A detailed discussion and verification of
most of these properties may be found in \cite{rao}. We are concerned here with linear operators acting on $L^p(\mathcal{F})$.
We continue our investigations about the class of bounded linear operators on the $L^p$-spaces having the form $M_wEM_u$, where $E$ is the conditional expectation operator, $M_w$ and $M_u$ are (possibly unbounded) multiplication operators and it is called weighted conditional expectation operator (WCE operator). Our interest in WCE operators stems from the fact that such forms tend to appear often in the study of those operators related to conditional expectation. WCE operators appeared in \cite{do}, where it is shown that every contractive projection on certain $L^1$-spaces can be decomposed into an operator of the form $M_wEM_u$ and a nilpotent operator. For more strong results about WCE operators one can see \cite{dhd,gb,lam,mo}, in these papers one can see that a large classes of operators are of the form of WCE operators.\\ In this paper we consider weighted conditional expectation operators of the form of $M_wEM_u$ on $L^p(\mathcal{F})$ and we find some results about the ascent and descent of these operators.

\section{ Main Results}
In this section first we give the definition of weighted conditional expectation operators on $L^p$-spaces.
\begin{defn}
Let $(X,\mathcal{F},\mu)$ be a $\sigma$-finite measure space and let $\mathcal{A}$ be a
$\sigma$-subalgebra of $\mathcal{F}$ such that $(X,\mathcal{A},\mu_{\mathcal{A}})$ is also $\sigma$-finite. Let $E$ be the conditional
expectation operator relative to $\mathcal{A}$. If $1\leq p,q\leq\infty$ and $u,w \in L^0(\mathcal{F})$ such that $uf$ is conditionable and $wE(uf)\in L^{q}(\mathcal{F})$ for all $f\in \mathcal{D}\subseteq L^{p}(\mathcal{F})$, where $\mathcal{D}$ is a linear subspace, then the corresponding weighted conditional expectation (or briefly WCE) operator is the linear transformation $M_wEM_u:\mathcal{D}\rightarrow L^{q}(\mathcal{F})$ defined by $f\rightarrowtail wE(uf)$.
\end{defn}
As was proved in \cite{ej}, the WCE operator $M_wEM_u$ on $L^p(\mathcal{F})$ is bounded if and only if $(E(|u|^{p'}))^{\frac{1}{p'}}(E(|w|^p))^{\frac{1}{p}}\in
L^{\infty}(\mathcal{A})$, where $\frac{1}{p}+\frac{1}{p'}=1$. Also $M_wEM_u$ on $L^1(\mathcal{F})$ is bounded if and only if $uE(|w|)\in L^{\infty}(\mathcal{F})$. Here we give a formula for $n$-powers of the WCE operator $T=M_wEM_u$ in the next lemma. The proof is omitted and it's an easy exercise.
\begin{lem}\label{l1} Let $n\in \mathbb{N}$ and $T=M_wEM_u$ be a bounded operator on $L^p(\mathcal{F})$. Then we have
$$T^n=M_{(E(uw))^{n-1}}M_{w}EM_{u}.$$
\end{lem}
Hence by using Lemma \ref{l1} we get the Cesaro mean for bounded WCE operator $T=M_wEM_u$ on $L^p(\mathcal{F})$ as follows:
$$A_n(T)=n^{-1}(I+M_{v_{n}}T), \ \ \ \ \ \ \forall n\in \mathbb{N}\setminus\{1\}$$
and $A_{1}(T)=I$, in which $v_n=\sum^{n-2}_{i=0}E(uw)^{i}$.\\
From now on we assume that $T=M_wEM_u$ is a bounded operator on $L^p(\mathcal{F})$.

\begin{thm}\label{t3} For $T=M_wEM_u$  we have
$$\mathcal{N}(T^2)=\mathcal{N}(T^{n+2}),$$
for all $n\in \mathbb{N}$. Moreover, $\alpha(T)\leq2$.
\end{thm}
\begin{proof}
Let $f\in L^p(\mathcal{F})$. Then we have
$$T^2(f)=E(uw)T(f)=0 \Leftrightarrow S(E(uw))\cap S(Tf)=\emptyset$$
Also,
$$T^{n+2}(f)=E(uw)^{n+1}T(f)=0 \Leftrightarrow S(E(uw)^{n+1})\cap S(Tf)=\emptyset.$$
Since $S(E(uw))=S(E(uw)^{n+1})$, then we get that $\mathcal{N}(T^2)=\mathcal{N}(T^{n+2})$.\\
\end{proof}
In the next theorem we decompose $L^p(\mathcal{F})$ as a direct sum of two its closed subspaces under some weak conditions.
\begin{thm}
If there exists an $n>2$ such that $\mathcal{R}(M_{E(uw)^{n-1}}T)$ is closed or $\mathcal{R}(M_{E(uw)^{n-2}}T)+\mathcal{N}(T)$ is closed or $\mathcal{R}(M_{E(uw)^{n-3}}T)+\mathcal{N}(M_{E(uw)}T)$ is closed, then  we have the followings:\\
a) $\mathcal{R}(M_{E(uw)^{n-1}}T$ is closed for all $n\geq2$ and $\mathcal{R}(M_{E(uw)^{j-1}}T)+\mathcal{N}(M_{E(uw)^{k-1}}T)$ is closed for all $j+k\geq2$.\\
b) $L^p(\mathcal{F})=\mathcal{R}(M_{E(uw)}T)\oplus\mathcal{N}(M_{E(uw)}T)$.\\
\end{thm}
\begin{proof} a) Since we have $T^n=M_{E(uw)^{n-1}}T$  and $\mathcal{N}(T^2)=\mathcal{N}(T^{n+2})$ for all $n\in \mathbb{N}$ by the Lemma \ref{l1} and Theorem \ref{t3}, respectively, then by the theorem 2.1 of \cite{gz} we get the proof.\\
b) Since for $T=M_wEM_u$ we have $\alpha(T)=\alpha(T^*)\leq2$ and $\alpha(T)=\delta(T^*)$, then we have the result.
\end{proof}
Here we find some necessary and sufficient conditions for $T=M_wEM_u$ to be power bounded .
\begin{thm}\label{t8} For $T=M_wEM_u\in\mathcal{B}(L^p(\mathcal{F}))$ we have the followings:\\
a) The sequence $\{\|E(uw)^n\|_{\infty}\}_{n\in \mathbb{N}}$ is uniformly bounded if and only if $\|E(uw)\|_{\infty}\leq 1$.\\
b) $T$ is power bounded if and only if $|E(wu)|<1$ on $S((E(|w|^p))^{\frac{1}{p}}(E(|u|^{p'}))^{\frac{1}{p'}})$.
\end{thm}
\begin{proof} a) Since $\|E(uw)^n\|_{\infty}=\|E(uw)\|^n_{\infty}$, then clearly we have that the sequence $\{\|E(uw)^n\|_{\infty}\}_{n\in \mathbb{N}}$ is uniformly bounded if and only if $\|E(uw)\|_{\infty}\leq 1$.\\

b) Let $|E(wu)|<1$ on $S((E(|w|^p))^{\frac{1}{p}}(E(|u|^{p'}))^{\frac{1}{p'}})$. Since $S(E(uw))\subseteq S((E(|w|^p))^{\frac{1}{p}}(E(|u|^{p'}))^{\frac{1}{p'}})$, then $|E(wu)|<1$ on $X$. Hence $\|E(w)\|_{\infty}\leq 1$ and so the sequence  $\{\|E(uw)^n\|_{\infty}\}_{n\in \mathbb{N}}$ is uniformly bounded. Therefore there exists $C>0$ such that $\|E(uw)^n\|_{\infty}<C$, for all $n\in \mathbb{N}$. Moreover, by Lemma \ref{l1} we have $T^n=M_{E(uw)^{n-1}}T$. Hence for every $f\in L^p(\mathcal{F})$ and  $n\in \mathbb{N}$ we have
$$\|T^nf\|_p=\|E(uw)^{n-1}T(f)\|_p\leq \|E(uw)^{n-1}\|_{\infty}\|T\|\|f\|_p\leq C\|T\|\|f\|_p.$$
This means that $T$ is power bounded.\\
Conversely, let $T$ be power bounded. Then we can find $C>0$ such that $\|M_{(E(uw))^{n-1}}T\|\leq C$, for all $n\in \mathbb{N}$. We know that $S(Tf)\subseteq S((E(|w|^p))^{\frac{1}{p}}(E(|u|^{p'}))^{\frac{1}{p'}})$ and $S(E(uw))\subseteq S((E(|w|^p))^{\frac{1}{p}}(E(|u|^{p'}))^{\frac{1}{p'}})$. Suppose that there exists $A\in \mathcal{F}$ with $\mu(A)>0$ such that $|E(uw)|>1$ on $A$. Then $\|E(uw)\|_{\infty}>1$ and so $\|(E(uw))^n\|_{\infty}\rightarrow \infty$. In this case $M_{(E(uw))^{n-1}}T$ is bounded if and only if $T=0$. So if $T\neq 0$, then we get a contradiction. Therefore we should have $|E(wu)|<1$ on $S((E(|w|^p))^{\frac{1}{p}}(E(|u|^{p'}))^{\frac{1}{p'}})$.
\end{proof}
In the next theorem we give some weak conditions for $T=M_wEM_u$ under which it has finite descent.
\begin{thm} Consider $T=M_wEM_u$ on $L^p(\mathcal{F})$ and suppose that $E(uw)$ is bounded away from zero i.e, there exists some $C>0$ such that $\mu(\{x\in X:E(uw)(x)\leq C\})=0$. Then we have
$$\mathcal{R}(T^2)=\mathcal{R}(T^{n+2})$$
for all $n\in \mathbb{N}$.
\end{thm}
\begin{proof}
Let $g\in \mathcal{R}(T^2)$ i,e $T^2f=g$ for $f\in L^p(\mathcal{F})$ and $G=S(E(uw)$. Since $T$ is bounded, then $(E(|w|^p))^{\frac{1}{p}}(E(|u|^q))^{\frac{1}{q}}\in L^{\infty}(\mathcal{A})$. Hence by conditional type Holder inequality we get that $E(uw)\in L^{\infty}(\mathcal{A})$. Therefore we will have

 \begin{align*}
 g&=E(uw)T(f)\\
 &=E(uw)^{n+1}E(uw)^{-n}\chi_{G}T(f)\\
 &=E(uw)^{n+1}T(E(uw)^{-n}\chi_{G}f)\\
 &=T^{n+2}(E(uw)^{-n}\chi_{G}f).
 \end{align*}
 Since $E(uw)$ is bounded away from zero, then we get that $E(uw)^{-n}\chi_{G}f\in L^p(\mathcal{F})$. And so $g\in \mathcal{R}(T^{n+2})$.
\end{proof}
\begin{cor} For the bounded operator $T=M_wEM_u$ on $L^p(\mathcal{F})$ we have $Ascent(T)\leq2$ and if $E(uw)$ is bounded away from zero, then
$\delta(T)\leq2$.
\end{cor}
\begin{cor} For $T=M_wEM_u$ on $L^p(\mathcal{F})$ the followings hold:
\begin{enumerate}
  \item $\mathcal{R}(T^2)\cap \mathcal{N}(T^{n+2})=\{0\}$, for all $n\in \mathbb{N}$.\\
  \item If $E(uw)$ is bounded away from zero, then $\mathcal{R}(T^{n+2})+ \mathcal{N}(T^{2})=L^p(\mathcal{F})$, for all $n\in \mathbb{N}$.\\
\end{enumerate}
\end{cor}

\begin{cor}
 Let $|E(wu)|<1$ on $S((E(|w|^p))^{\frac{1}{p}}(E(|u|^{p'}))^{\frac{1}{p'}})$ and $T=M_wEM_u$ on $L^p(\mathcal{F})$ for $1<p<\infty$. Then $\alpha(I-T)\leq 1$ and  $\alpha(I-T^{\ast})\leq1$.
\end{cor}
\begin{proof} By using Theorem \ref{t8} we get that $T$ is power bounded. And so by the Lemma 3.1 and  Theorem 3.2 of \cite{gz} we get the result.
\end{proof}
In the sequel we find that the range and kernel of $T^2$ are quasi complements in $L^p(\mathcal{F})$ under some weak conditions.
\begin{thm} If $T=M_wEM_u$ and $\mathcal{R}(T^2)$  is closed, then $\mathcal{R}(T^2)$ and $\mathcal{N}(T^2)$ are quasi-complements, it means that $\mathcal{R}(T^2)\cap\mathcal{N}(T^2)=0$ and $\mathcal{R}(T^2)+\mathcal{N}(T^2)$ is dense in $L^p(\mathcal{F})$.
\end{thm}
\begin{proof} By Theorem \ref{t3} we have that $\alpha(T)=\alpha(T^*)\leq2$. Therefore by Theorem 4.1 of \cite{gz} we get that $\mathcal{R}(T^2)$ and $\mathcal{N}(T^2)$ are quasi-complements.
\end{proof}
By using some results of \cite{gz} for the power bounded operator $T=M_wEM_u$ we can write $L^p(\mathcal{F})$ as the direct some of two closed subspaces.
\begin{pro} If $|E(wu)|<1$ on $S((E(|w|^p))^{\frac{1}{p}}(E(|u|^{p'}))^{\frac{1}{p'}})$ and $\mathcal{R}(I-T)$ is closed, then $L^p(\mathcal{F})=\mathcal{R}(I-T)\oplus \mathcal{N}(I-T)$.
\end{pro}
\begin{proof} If the sequence $\{\|E(uw)^n\|_{|\infty}\}_{n\in \mathbb{N}}$ is uniformly bounded, then $n^{-1}M_{E(uw)^{n-1}}Tf\rightarrow 0$ for every $f\in L^p(\mathcal{F})$. Hence by Theorem 4.4 of \cite{gz}  we get the proof.
\end{proof}
We know that the spectral radius of $T$ is equal to $\|E(uw)\|_{\infty}$ \cite{y1}. So if we assume $\|E(uw)\|_{\infty}\leq 1$, then we have $(I-T)^{-1}=\lim_{n \rightarrow \infty} (I+M_{v_n}T)$. Moreover, we have computed the inverse $(I-T)^{-1}$ for $T+M_wEM_u$ in \cite{ye}. Here we have a formula for $(I-T)^{-1}$ under some weaker conditions.
\begin{rem}
Let $|E(wu)|<1$ on $S((E(|w|^p))^{\frac{1}{p}}(E(|u|^{p'}))^{\frac{1}{p'}})$ and $T=M_wEM_u$ . Then $\overline{\mathcal{R}(I-T)}=L^p(\mathcal{F})$ and equivalently we have the followings:\\
(i) $\mathcal{R}(I-T)=L^p(\mathcal{F})$;\\
(ii) $I-T$ is invertible;\\
(iii) $\{\|(I-T)^{-1}(\frac{(n-1)I-M_{v_n}T}{n})\|\}_{n\in \mathbb{N}}=\{\|n^{-1}(M_{w_n}T+(n-1)I)\|\}_{n\in \mathbb{N}}$ is bounded;\\
(iv) $\{(I-T)^{-1}(\frac{(n-1)I-M_{v_n}T}{n})(f)\}_{n\in \mathbb{N}}=\{n^{-1}(M_{w_n}Tf+(n-1)f)\}_{n\in \mathbb{N}}$  converges for all $f\in L^p(\mathcal{F})$. In which $w_n=\sum^{n-2}_{i=1}(n-i-1)E(uw)^{i-1}$.\\
In this case, $n^{-1}(M_{w_n}Tf+(n-1)f)\rightarrow (I-T)^{-1}(f)$ for all $f\in L^p(\mathcal{F})$.\\
(v) $A_n(T)(f)$ converges to a $T$-invariant limit for all $f\in L^p(\mathcal{F})$.
\end{rem}
\begin{proof} By our assumption we get that $T$ is power bounded and so for every $f\in L^p(\mathcal{F})$ we have $A_n(T)(f)\rightarrow 0$. As is defined in \cite{gz} $B_n(T)=n^{-1}(T^{n-2}+2T^{n-3}+....+(n-2)T+(n-1)I)$. we have it for $T=M_wEM_u$ as follows:
$$B_n(T)=n^{-1}(M_{w_n}T+(n-1)I),$$
 in which $w_n=\sum^{n-2}_{i=1}(n-i-1)E(uw)^{i-1}$.
 Therefore by Proposition 4.5 of \cite{gz} we have the proof of (i)-(iv). Since $\{\|E(uw)^n\|_{\infty}\}_{n\in \mathbb{N}}$ is uniformly bounded, then $T=M_wEM_u$ is power bounded. Also the closure of a norm bounded subset of $L^p(\mathcal{F})$ is weakly compact. Hence it's weakly sequentially compact. Then we have (v).
\end{proof}
Here we recall a fundamental lemma in operator theory.
\begin{lem}\label{l2} Let $T$ be a bounded operator on the Hilbert space $\mathcal{H}$ and $\lambda\geq0$. Then we have
$$\|\lambda I+T^{\ast}T\|=\lambda+\|T^{\ast}T\|=\lambda+\|T\|^2.$$
Specially, if $T$ is a positive operator, then
$\|\lambda I+T\|=\lambda+\|T\|$.
\end{lem}
\begin{proof} It is an easy exercise.
\end{proof}
If we set $u=\bar{w}$, then $T=M_wEM_u$ is positive on $L^2(\mathcal{F})$ and so $A_n(T)=n^{-1}(I+M_{v_n}M_{w}EM_{\bar{w}})$ is positive too, in which $v_n=\sum^{n-2}_{i=0}E(|w|^2)^i$. By using the Lemma \ref{l2} we get that $\|A_n(T)\|=n^{-1}(1+\|v_nE(|w|^2)\|_{\infty})$. So $T$ is Cesaro bounded if and only if the sequence $\{n^{-1}\|v_nE(|w|^2)\|_{\infty}\}_{n\in \mathbb{N}}$ is uniformly bounded.
One can see the paper \cite{ej} for more information about $T=M_wEM_u$. For instance we have the Aluthge transformation of
$T=M_wEM_u$ on $L^2(\mathcal{F})$ as follow:
$$\widehat{T}(f)=\frac{\chi_{S}E(uw)}{E(|u|^{2})}\bar{u}E(uf), \ \ \ \ \ \  \ \ \ \  \ \  \ f\in L^{2}(\Sigma),$$
in which $S$ is the support of $E(|u|^{2})$. Therefore the Aluthge transformation of $T$ is a weighted conditional expectation operator $\widehat{T}=M_vEM_u$, in which $v=\frac{\chi_{S}E(uw)}{E(|u|^{2})}\bar{u}$. Hence we have $\|\widehat{T}\|=\|E(uw)\|_{\infty}$, $E(uv)=E(uw)$ and $\widehat{T}^n=M_{\frac{\chi_{S}E(uw)^n}{E(|u|^{2})}}M_{\bar{u}}EM_u$. In addition we have $\|\widehat{T}^n\|=\|E(uw)^n\|_{\infty}$ and $v_n=\sum^{n-2}_{i=0}E(\frac{\chi_{S}E(uw)^n}{E(|u|^{2})}\bar{u}u)^i=\sum^{n-2}_{i=0}E(uw)^i$. By these observations we have the next theorem.
\begin{thm} For the bounded operator $T=M_wEM_u$ with $w,u\geq 0$, the followings are equivalent:\\
a) $T$ is Cesaro bounded.\\
b) $\widehat{T}$ is Cesaro bounded.\\
c) The sequence $\{n^{-1}\|v_n\|_{\infty}\}_{n\in \mathbb{N}}$ is uniformly bounded.\\
d) The sequence $\{n^{-1}\|v_n(E(|w|^2))^{\frac{1}{2}}(E(|w|^2))^{\frac{1}{2}}\|_{\infty}\}_{n\in \mathbb{N}}$ is uniformly bounded.
\end{thm}

%

\end{document}